\documentclass[a4paper,10pt]{amsart}
\usepackage{pst-all, pstricks, pst-node}
\usepackage{subfigure}
\usepackage{graphicx, amssymb, amsmath, amscd, hhline}
\usepackage[arrow,tips,matrix]{xy}

\theoremstyle{ams}
\newtheorem{theorem}{Theorem}[section]
\newtheorem{proposition}[theorem]{Proposition}
\newtheorem{lemma}[theorem]{Lemma}
\newtheorem{corollary}[theorem]{Corollary}

\theoremstyle{definition}
\newtheorem{question}[theorem]{Question}

\newtheorem{definition}[theorem]{Definition}

\newcommand{\C}{\mathbb{C}}

\newcommand{\R}{\mathbb{R}}

\begin{document}
\title[Unimodality of Betti numbers]{Unimodality of Betti numbers for Hamiltonian circle actions with index-increasing moment maps}

\author[Y. Cho]{Yunhyung Cho}
\address{Departamento de matem\'{a}tica, Centro de An\'{a}lise Matem\'{a}tica, Geometria e Sistemas
Din\^{a}micos-LARSYS, Instituto Superior T\'{e}cnico, Av. Rovisco Pais 1049-001 Lisbon, Portugal}

\email{yhcho@kias.re.kr, ~yuncho@math.ist.utl.pt}

%\date{\today}
\maketitle

\begin{abstract}

The unimodality conjecture posed by Tolman in \cite{JHKLM} states that if $(M,\omega)$ is a $2n$-dimensional smooth compact symplectic manifold equipped with a Hamiltonian circle action with only isolated fixed points, then the sequence of Betti numbers $\{b_0(M), b_2(M), \cdots\}$ is unimodal, i.e. $b_i(M) \leq b_{i+2}(M)$ for every $i < n$.
Recently, the author and M. Kim \cite{CK} proved that the unimodality holds in eight-dimensional cases by using equivariant cohomology theory.
In this paper, we generalize the idea in \cite{CK} to an arbitrary dimensional case. Also, we prove the conjecture in arbitrary dimension with an assumption that a moment map $H : M \rightarrow \R$ is \textit{index-increasing}, which means that $\mathrm{ind}(p) < \mathrm{ind}(q)$ implies $H(p) < H(q)$ for every pair of critical points $p$ and $q$ of $H$ where $\mathrm{ind}(p)$ is a Morse index of $p$ with respect to $H$.

\end{abstract}

\section{introduction}

Let $\mathcal{A} = \{a_1,\cdots,a_n\}$ be a finite sequence of real numbers. We say that $\mathcal{A}$ is \textit{unimodal} if there is a positive integer $k$ (called a \textit{mode} of $\mathcal{A}$) such that $a_i \leq a_{i+1}$ for every $i<k$ and $a_j \geq a_{j+1}$ for every $j \geq k$. Similarly, we say that a polynomial $\sum_i a_ix_i \in \R[x]$ is \textit{unimodal} if $\{a_i \}$ is unimodal.

In K\"{a}hler geometry, it is well-known that a sequence of even (odd, respectively) Betti numbers is unimodal. More precisely, let $(M,\omega,J)$ be a complex $n$-dimensional compact K\"{a}hler manifold. Then the hard Lefschetz theorem says that
\begin{displaymath}
\begin{array}{cccc}
[\omega]^{n-k} : & H^{k}(M;\R) & \longrightarrow &
H^{2n-k}(M;\R)\\ %[0.5em]
& \alpha & \mapsto & \alpha
\wedge [\omega]^{n-k}
\end{array}
\end{displaymath}
is an isomorphism for every $k=0, 1, \cdots, n$. In particular, the map $$H^k(M;\R) \stackrel{\wedge [\omega]} \longrightarrow H^{k+2}(M;\R)$$ is injective for every $k<n$. Hence
the sequences $\{b_0(M), b_2(M), \cdots, b_{2n}(M)\}$ and $\{b_1(M), b_3(M), \cdots, b_{2n-1}(M) \}$ are unimodal by Poincar\'{e} duality. In the symplectic category, there are a lot of examples which do not satisfy the hard Lefschetz theorem so that the unimodality of Betti numbers is not obvious in general. (See \cite{BG}, \cite{Cho}, and \cite{Gom}).

In the conference ``Moment maps in various geometries" in 2005, the unimodality of Betti numbers of symplectic manifolds was discussed in the equivariant setting (See \cite{JHKLM}). More precisely, let $S^1$ be the unit circle group acting on a symplectic manifold $(M,\omega)$ in a Hamiltonian fashion. S. Tolman posed the following question.

%Let $\frak{t}$ be the Lie algebra of $S^1$ with a generator $X \in \frak{t}$, and let $\underline{X}$ a vector field on $M$ generated by $X$.
%The $S^1$-action on $(M,\omega)$ is called \textit{symplectic} if it preserves the symplectic form $\omega$. It is equivalent to saying that the $S^1$-action is symplectic if and only if
%the interior product $i_{\underbar{X}} \omega$ is closed. If $i_{\underbar{X}} \omega$ is exact, then we say that the action is \textit{Hamiltonian} and a smooth function $H : M \rightarrow \R$ satisfying $i_{\underbar{X}} \omega = -dH$ is called a \textit{moment map}.

\begin{question}\cite{JHKLM} \label{question : Tolman}
Let $(M,\omega)$ be a $2n$-dimensional closed symplectic manifold equipped with a Hamiltonian circle action with only isolated fixed points. Then is $\{b_0(M), b_2(M), \cdots, b_{2n}(M) \}$ unimodal?
\end{question}

As far as the author knows, Question \ref{question : Tolman} was originated from the existence problem of non-K\"{a}hler symplectic manifold with a Hamiltonian circle action with only isolated fixed points. In other words, there has not been found any example of compact Hamiltonian $S^1$-manifold with only isolated fixed points which is NOT homotopy equivalent to any K\"{a}hler manifold. Since the problem of the existence of K\"{a}hler structure is essentially related to the classification problem, it is extremely hard to deal with in general.
Hence it seems more reasonable to investigate whether our manifold satisfies conditions with which any K\"{a}hler manifold satisfies.
Note that there are several simple topological obstructions for the existence of K\"{a}hler structure, such as the evenness of odd Betti numbers $b_{2i+1}(M)$ induced by the Hodge symmetry, and the unimodality of Betti numbers induced by the hard Lefschetz theorem.

There is another point of view to consider Question \ref{question : Tolman}. By Delzant's theorem \cite{De}, any $2n$-dimensional compact symplectic manifold $(M,\omega)$ equipped with an effective Hamiltonian $T^n$-action is a projective toric variety so that the sequence of Betti numbers is unimodal by the hard Lefschetz theorem.
Since the action is effective, the fixed point set $M^{T^n}$ must be discrete. Hence for a generic choice of a circle subgroup $S^1 \subset T^n$, the induced Hamiltonian circle action has only isolated fixed points. Consequently, we may regard Question \ref{question : Tolman} as a generalization problem of the unimodality of Betti numbers for Hamiltonian $T^n$-actions to the case of Hamiltonian $S^1$-action.

Recently, the author and M. Kim \cite{CK} proved that the unimodality holds in eight-dimensional cases.
In this paper, we generalized the idea used in \cite{CK} and proved the unimodality of Betti numbers for Hamiltonian circle action with isolated fixed points with some extra condition. More precisely, let $H : M \rightarrow \R$ be a moment map for a Hamiltonian $S^1$-manifold $(M,\omega)$ with isolated fixed point set $M^{S^1}$. Then it is well-known that $H$ is a perfect Morse function whose critical point set of $H$ equals to $M^{S^1}$. We denote by $\mathrm{ind}(p)$ a Morse index of a critical point $p$ of $H$.
We say that $H$ is \textit{index-increasing} if
$$ \mathrm{ind}(p) < \mathrm{ind}(q) \Rightarrow  H(p) < H(q)$$
for every $p$ and $q$ in $M^{S^1}$. Our main result is as follows.

    \begin{theorem}\label{theorem : main}
        Let $(M,\omega)$ be a $2n$-dimensional closed symplectic manifold equipped with a Hamiltonian circle action with only isolated fixed points. Let $H : M \rightarrow \R$ be a moment map which is index-increasing.
        Then the sequence of even Betti numbers of $M$ is unimodal.
    \end{theorem}
Our idea is as follows.
Let $(M,\omega)$ be a $2n$-dimensional compact Hamiltonian $S^1$-manifold with only isolated fixed points $M^{S^1} = \{F_1,\cdots,F_m\}$.
Since the corresponding moment map is perfect Morse, the number of fixed points of index $2i$ equals to $b_{2i}(M)$.
Now, let's consider the decomposition of equivariant cohomology
$$H^*_{S^1}(M) = \mathcal{R}_0 \oplus \mathcal{R}_2 \oplus \cdots$$ where $\mathcal{R}_{2i}$ is the set of all elements of degree $2i$ in $H^*_{S^1}(M)$.
The main point of our proof is, for any $I \subset [m] = \{1,\cdots,m\}$ with $ |I| = k < \dim_{\R} \mathcal{R}_{2i}$, there must exist an element $\alpha \in \mathcal{R}_{2i}$
such that the restriction of $\alpha$ on the fixed point set $\{ F_{i_j} ~|~ i_j \in I\}$ vanishes. This fact is obvious, since the restriction map
\begin{displaymath}
    \begin{array}{llccc}
        r & : & \mathcal{R}_{2i} & \rightarrow & \bigoplus_{i \in I} H^{2i}_{S^1}(F_i) \cong \R^k \\[0.5em]
          &   & \alpha           & \mapsto     & (\alpha|_{F_{i_1}}, \cdots, \alpha|_{F_{i_k}}) \\
    \end{array}
\end{displaymath}
is $\R$-linear and it must have non-trivial kernel by dimensional reason. Here, the restriction $\alpha|_F$ means $i^*_F(\alpha)$ where
$$ i^*_F : H^*_{S^1}(M) \rightarrow H^*_{S^1}(F) \cong \R[u] $$
is a ring homomorphism induced by the inclusion $i_F : F \hookrightarrow M$.
Hence if unimodality fails, i.e. if there is some positive integer $i < n$ such that $b_{2i}(M) > b_{2i+2}(M)$, then there must be some element $\alpha \in H^*_{S^1}(M)$ such that
the restriction $\alpha|_z$ vanishes for every fixed point $z$ of index $k$ for $k < 2i$ or $k = 2i+2$ since $b_0(M) + \cdots + b_{2i-2}(M) + b_{2i}(M) > b_0(M) + \cdots + b_{2i-2}(M) + b_{2i+2}(M)$ by our assumption. In Section 3, we will show that the existence of such $\alpha$ leads to a contraction with Atiyah-Bott-Berline-Vergne localization theorem \ref{theorem : localization}.

This paper is organized as follows.
In Section 2, we give a brief introduction to equivariant cohomology theory for Hamiltonian circle actions, including Kirwan's injectivity theorem and Atiyah-Bott-Berline-Vergne localization theorem for circle actions. In Section 3, we give the complete proof of Theorem \ref{theorem : main}.

\section{Equivariant cohomology}

In this section, we give a brief introduction to equivariant cohomology theory for a Hamiltonian circle action on a symplectic manifold.
Throughout this section, we assume that $(M,\omega)$ is a $2n$-dimensional smooth compact symplectic manifold equipped with a Hamiltonian $S^1$-action with a moment map $H : M \rightarrow \R$. Also, every coefficient of any cohomology theory is assumed to be the field of real numbers $\R$.
\subsection{Equivariant cohomology.}An equivariant cohomology $H^*_{S^1}(M)$ is defined by
                                $$ H^*_{S^1}(M) = H^*(M \times_{S^1} ES^1) $$ where $ES^1$ is a contractible space on which $S^1$ acts freely. In particular, the equivariant cohomology of
                                a point $p$ is $$H^*_{S^1}(p) = H^*(p \times_{S^1} ES^1) = H^*(BS^1) $$ where $BS^1 = ES^1 / S^1$ is the classifying space of $S^1$. Note that $BS^1$ can be constructed as an inductive limit of the sequence of Hopf fibrations
                                \begin{equation}
                                    \begin{array}{ccccccccc}
                                        S^3          & \hookrightarrow & S^5        & \hookrightarrow & \cdots & S^{2n+1} & \cdots & \hookrightarrow & ES^1 \sim S^{\infty} \\
                                        \downarrow   &                 & \downarrow &                 & \cdots & \downarrow & \cdots &                 & \downarrow \\
                                        \C P^1       & \hookrightarrow & \C P^2     &\hookrightarrow  & \cdots & \C P^n      & \cdots & \hookrightarrow &BS^1 \sim \C P^{\infty}
                                    \end{array}
                                \end{equation}
                                Hence we have $$H^*(BS^1) = \R[u]$$ where $u$ is an element of degree two with $\langle u, [\C P^1] \rangle = 1$.

\subsection{Equivariant formality} Note that a projection map $M \times ES^1 \rightarrow ES^1$ on the second factor is $S^1$-equivariant so that it induces a map
                                       $$ \pi : M \times_{S^1} ES^1 \rightarrow BS^1 $$  which makes $M \times_{S^1} ES^1$ into an $M$-bundle over $BS^1$
                                       \begin{equation}\label{diagram : M-bundle over BS^1}
                                            \begin{array}{ccc}
                                                M \times_{S^1} ES^1 & \stackrel{f} \hookleftarrow & M \\[0.3em]
                                                \pi \downarrow          &                             &   \\[0.3em]
                                                BS^1                   &                             &
                                            \end{array}
                                       \end{equation}
                                       where $f$ is an inclusion of $M$ as a fiber. So it induces the following sequence of ring homomorphisms
                                       $$ H^*(BS^1) \stackrel{\pi^*} \rightarrow H^*_{S^1}(M) \stackrel{f^*} \rightarrow H^*(M). $$
                                       In particular, $H^*_{S^1}(M)$ has an $H^*(BS^1)$-module structure via the map $\pi^*$ such that
                                       $$ u \cdot \alpha = \pi^*(u)\cup \alpha $$ for $u \in H^*(BS^1)$ and $\alpha \in H^*_{S^1}(M)$.
                                       In our situation, the equivariant cohomology of Hamiltonian circle action has a remarkable property as follows.
                                       \begin{theorem}\label{theorem : equivariant formality}\cite{Ki}
                                            Let $(M,\omega)$ be a smooth compact symplectic manifold equipped with a Hamiltonian circle action. Then $M$ is equivariatly formal, that is,
                                            $H^*_{S^1}(M)$ is a free $H^*(BS^1)$-module so that $$H^*_{S^1}(M) \cong H^*(M) \otimes H^*(BS^1).$$
                                            Equivalently, the map $f^*$ is surjective with kernel $u \cdot H^*_{S^1}(M)$ where $\cdot$ means the scalar multiplication of $H^*(BS^1)$-module structure on $H^*_{S^1}(M)$.
                                       \end{theorem}

\subsection{Localization theorem} Let $\alpha \in H^*_{S^1}(M)$ be any element of degree $k$. Then Theorem \ref{theorem : equivariant formality} implies that $\alpha$ can be uniquely
                                      expressed as $$ \alpha = \alpha_k \otimes 1 + \alpha_{k-2} \otimes u + \alpha_{k-4} \otimes u^2 + \cdots $$
                                      where $\alpha_i \in H^i(M)$ for each $i = k, k-2, \cdots$. By Theorem \ref{theorem : equivariant formality}, we have $f^*(\alpha) = \alpha_k$.

                                      \begin{definition}
                                         An \textit{integration along the fiber $M$} is an $H^*(BS^1)$-module homomorphism $\int_M : H^*_{S^1}(M) \rightarrow H^*(BS^1)$ defined by
                                         $$ \int_M \alpha = \langle \alpha_k, [M] \rangle \cdot 1 + \langle \alpha_{k-2}, [M] \rangle \cdot u + \cdots $$
                                         for every $ \alpha = \alpha_k \otimes 1 + \alpha_{k-2} \otimes u + \alpha_{k-4} \otimes u^2 + \cdots \in H^k_{S^1}(M).$ Here, $[M]$ is the fundamental homology class of $M$.
                                      \end{definition}
                                       Note that $\langle \alpha_i, [M] \rangle$ is zero for every $i < \dim M = 2n$, and $\alpha_i = 0$ for every $i > \deg \alpha$ by dimensional reason. Hence we have $\int_M \alpha = \langle \alpha_{2n},[M] \rangle u^{k-2n}$. In particular, if $\deg \alpha < \dim M$, then we have
                                            $$ \int_M \alpha = 0 \in H^*(BS^1).$$
                                      Now, let $F \subset M^{S^1}$ be a fixed component with an inclusion map $i_F : F \hookrightarrow M$. Then it induces a ring homomorphism $$i_F^* : H^*_{S^1}(M) \rightarrow H^*_{S^1}(F) \cong H^*(F) \otimes H^*(BS^1).$$
                                      \begin{theorem}\cite[Kirwan's injectivity theorem]{Ki}\label{theorem : Kirwan injectivity}
                                          Let $(M,\omega)$ be a compact Hamiltonian $S^1$-manifold. For an inclusion $i : M^{S^1} \hookrightarrow M$, the induced map
                                          $$ i^* : H^*_{S^1}(M) \rightarrow H^*_{S^1}(M^{S^1})$$ is injective.
                                      \end{theorem}
                                      For any $\alpha \in H^*_{S^1}(M)$, we call an image $i_F^*(\alpha)$ \textit{the restriction of $\alpha$ to $F$} and denote by $\alpha|_F = i_F^*(\alpha)$ for simplicity. The following theorem due to Atiyah-Bott \cite{AB} and Berline-Vergne \cite{BV} enable us to compute $\int_M \alpha$ in terms of the fixed points.

                                      \begin{theorem}\label{theorem : localization}(Atiyah-Bott-Berline-Vergne localization)
                                          For any $ \alpha \in H^*_{S^1}(M)$, we have
                                          $$\int_M \alpha = \sum_{F \subset M^{S^1}} \int_F \frac{\alpha|_F}{e^{S^1}(F)} $$
                                          where $e^{S^1}(F)$ is the equivariant Euler class of the normal bundle of $F$.
                                      \end{theorem}
                                      In particular, if every fixed point is isolated, then we have the following corollary.

                                      \begin{corollary}\label{corollary : localization, isolated}
                                          Let $(M,\omega)$ be a $2n$-dimensional smooth compact symplectic manifold equipped with a Hamiltonian circle action with isolated fixed points. Let $\alpha \in H^*_{S^1}(M)$. Then we have
                                          $$\int_M \alpha = \sum_{F \in M^{S^1}} \frac{\alpha|_F}{\prod_i w_i(F)u}$$
                                          where the sum is taken over all fixed points, and $\{w_1(F), \cdots, w_n(F)\}$ is the weights of the tangential $S^1$-representation on $T_FM$.
                                      \end{corollary}

\subsection{Equivariant symplectic classes}     Let $H : M \rightarrow \R$ be a moment map for a closed Hamiltonian $S^1$-manifold $(M,\omega)$. For the product space $M \times ES^1$,
                                                consider a two form $\omega_H := \omega + d(H \cdot \theta)$, regarding $\omega$ as a pull-back
                                                of $\omega$ along the projection $M \times ES^1 \rightarrow M$ on the first factor and $\theta$ as a pull-back of a connection 1-form on the principal $S^1$-bundle $ES^1 \rightarrow BS^1$ along the projection
                                                $M \times ES^1 \rightarrow ES^1$ on the second factor. Here, the connection form $\theta$ on $ES^1$ is a finite dimensional approximation of the connection form of the principal $S^1$-bundle $S^{2n-1} \rightarrow \C P^n$. (See \cite{Au} for the details). It is not hard to show that $\mathcal{L}_{\underbar{X}} \omega_H = i_{\underbar{X}} \omega_H = 0$ where $\underbar{X}$ is the fundamental vector field on $M \times ES^1$ generated by the diagonal action. Hence we can push-forward $\omega_H$ to the quotient $M \times_{S^1} ES^1$ and denote by $\widetilde{\omega}_H$ the push-forward of $\omega_H$. Obviously, the restriction of $\widetilde{\omega}_H$ on each fiber $M$ is
                                                precisely $\omega$ and we call $\widetilde{\omega}_H$ \textit{the equivariant symplectic form with respect to $H$} and the corresponding
                                                cohomology class $[\widetilde{\omega}_H] \in H^2_{S^1}(M)$ is called \textit{the equivariant symplectic class with respect to $H$}.
                                                The restriction of the equivariant symplectic class to each fixed component can be easily computed as follows.
                                                \begin{proposition}\label{proposition : equivariant cohomology class}
                                                    Let $F \in M^{S^1}$ be a fixed component of the given Hamiltonian circle action. Then we have
                                                    $$[\widetilde{\omega}_H]|_F = [\omega]|_F \otimes 1 - H(F) \otimes u \in H^*(F) \otimes H^*(BS^1).$$
                                                    In particular, if $F$ is isolated, then we have $[\widetilde{\omega}_H]|_F = - H(F)u.$
                                                \end{proposition}
                                                \begin{proof}
                                                    Consider a push-forward of $\widetilde{\omega}_H|_F = (\omega - dH \wedge \theta - H \cdot d\theta)|_{F \times ES^1}$ to $F \times BS^1$. Since the restriction $dH|_{F \times ES^1}$ vanishes, we have $[\widetilde{\omega}_H]|_F = [\omega]|_F \otimes 1 - H(F) \otimes \widetilde{[d\theta]}|_{BS^1}$ where $\widetilde{[d\theta]}$ is a push-forward of $[d\theta]$ to $F \times BS^1$. Since the push-forward of $d\theta$ is a curvature form which represents the first Chern class of $ES^1 \rightarrow BS^1$, we have $\widetilde{[d\theta]} = -u$. Therefore $[\widetilde{\omega}_H]|_F = [\omega]|_F \otimes 1 - H(F) \otimes u.$
                                                \end{proof}

\section{Proof of Theorem \ref{theorem : main}}

For a Hamiltonian circle action on $(M,\omega)$ with isolated fixed points, recall that a moment map $H : M \rightarrow \R$ for the action is a perfect Morse function. For each fixed point $p \in M^{S^1}$, we denote by $\mathrm{ind}(p)$ a Morse index of $p$ with respect to $H$. We say that $H$ is \textit{index-increasing} if $$\mathrm{ind}(p)< \mathrm{ind}(q) ~\text{implies} ~H(p) < H(q)$$ for every pair of fixed points $p$ and $q$.

\begin{lemma}\label{lemma : dimension of H2kS^1}
    For each $2k \leq 2n$, a dimension of $H^{2k}_{S^1}(M)$ is $b_0 + b_2 + \cdots + b_{2k}$ where $b_i$ is the $i$-th Betti number of $M$.
\end{lemma}

\begin{proof}
    Since $M$ is equivariantly formal by Theorem \ref{theorem : equivariant formality}, we have 
    $$H^{2k}_{S^1}(M) \cong H^0(M) \otimes H^{2k}(BS^1) \oplus \cdots \oplus H^{2k}(M) \otimes H^0(BS^1).$$
    Hence it follows from $\dim_{\R} H^{2i}(BS^1) = 1$ for each $i$. 
     
\end{proof}

\begin{lemma}\label{lemma : vanishing lemma on points}
    Let $\mathcal{P} = \{p_1,\cdots,p_r\}$ be any subset of the fixed point set $M^{S^1}$ with $r < b_0 + b_2 + \cdots + b_{2k}$. Then there exists a non-zero class $\alpha \in H^{2k}_{S^1}(M)$
    such that $\alpha|_{p_i}=0$ for every $i=1,\cdots,r$.
\end{lemma}

\begin{proof}
    For any $\alpha \in H^{2k}_{S^1}(M)$ and $p \in M^{S^1}$, the restriction $\alpha|_p \in H^{2k}_{S^1}(p) \cong \R$ is a polynomial with a variable $u$ of degree $k$.
    Hence we have $\alpha|_{p} = au^k$ for some $a \in \R$.
    Let's consider the following map
    \begin{displaymath}
        \begin{array}{cccc}
            \phi_{2k}^{\mathcal{P}} : & H^{2k}_{S^1}(M;\R) & \longrightarrow & \R^r\\ %[0.5em]
                                      &     \alpha         & \mapsto         & (\alpha|_{p_1}, \cdots, \alpha|_{p_r})\\
        \end{array}
    \end{displaymath}
    Then $\phi_{2k}^{\mathcal{P}}$ is $\R$-linear and it has a non-trivial kernel by dimensional reason.
\end{proof}

Now, we are ready to prove our main theorem \ref{theorem : main}.

\begin{proof}[Proof of Theorem \ref{theorem : main}]

Let $(M,\omega)$ be a $2n$-dimensional symplectic manifold with a Hamiltonian circle action with only isolated fixed points. By our assumption, a moment map $H : M \rightarrow \R$ for the action is assumed to be \textit{index-increasing}. For each $i=0,1,\cdots,n$, we denote by $\Lambda_{2i}$ the set of all fixed points of index $2i$ so that we have
$$ b_{2i} := b_{2i}(M) = |\Lambda_{2i}| = |\Lambda_{2n-2i}|. $$
Now, let's assume that $b_{2k} > b_{2k+2}$ for some $0 \leq 2k < n$. Let
\begin{displaymath}
    \begin{array}{l}
        \mathcal{P}_1 = \cup_{i\geq 1} \Lambda_{2k-4i+2}, \\
        \mathcal{P}_2 = \cup_{i \geq 1} \Lambda_{2n-2k+4i},\\
        \mathcal{P}_3 = \Lambda_{2n-2k-2},\\
        \mathcal{P} = \mathcal{P}_1 \cup \mathcal{P}_2 \cup \mathcal{P}_3.
    \end{array}
\end{displaymath}
Since $2n - 2k - 2 > 2k - 2 \geq 2k - 4i + 2$ and $2n - 2k - 2 < 2n - 2k + 4i$ for every $i \geq 1$, $\mathcal{P}_i$'s are pairwise disjoint so that $\mathcal{P}$ is the disjoint union of $\mathcal{P}_1$. $\mathcal{P}_2$, and $\mathcal{P}_3$. Hence the cardinality of $\mathcal{P}$ is given by
\begin{displaymath}
    \begin{array}{lll}
        |\mathcal{P}|  & = |\mathcal{P}_1| + |\mathcal{P}_2| + |\mathcal{P}_3| \\
                       & = \sum_{i \geq 1} b_{2k-4i+2} + \sum_{i \geq 1} b_{2n-2k+4i} + b_{2n-2k-2}\\
                       & = \sum_{i \geq 1} b_{2k-4i+2} + \sum_{i \geq 1} b_{2k - 4i} + b_{2k+2} \\
                       & = b_0 + b_2 + \cdots + b_{2k-2} + b_{2k+2} < b_0 + b_2 + \cdots + b_{2k-2} + b_{2k}.\\
    \end{array}
\end{displaymath}
Hence there exists a non-zero class $\alpha \in H^{2k}_{S^1}(M)$ which vanishes on $\mathcal{P}$ by Lemma \ref{lemma : vanishing lemma on points}.
Let
\begin{displaymath}
    \begin{array}{lll}
        I_1  & = &\Lambda_0 \cup \Lambda_2 \cup \cdots \cup \Lambda_{2k} \\
        I_2  & = &\Lambda_{2k+2}\\
        I_3  & = &\Lambda_{2k+4}\\
        \vdots &  &\vdots\\
        I_{n-2k-2} & = & \Lambda_{2n-2k-6}\\
        I_{n-2k-1} & = & \Lambda_{2n-2k-4} \cup \Lambda_{2n-2k-2} \cup \Lambda_{2n-2k}\\
        I_{n-2k} & = & \Lambda_{2n-2k+2} \cup \Lambda_{2n-2k+4} \cup \cdots \cup \Lambda_{2n}\\
    \end{array}
\end{displaymath}
be a decomposition of the fixed point set $M^{S^1}$ into disjoint $(n-2k)$-subsets of $M^{S^1}$. Since $H$ is index-increasing, there exists $(n-2k-1)$ numbers
$\{ r_1, r_2, \cdots, r_{n-2k-1} \}$ such that
$$ H(p) < r_i < H(q) $$
for every $p \in I_i$ and $q \in I_{i+1}$. 
For each $i=1,\cdots,n-2k-1$, let $H_i = H - r_i$ be a new moment map and $[\omega_{H_i}]$ be the equivariant symplectic class with respect to $H_i$.
Let  $$\beta := \alpha^2 \cdot [\omega_{H_1}] \cdot [\omega_{H_2}] \cdots [\omega_{H_{n-2k-1}}] \in H^{2n-2}_{S^1}(M).$$
Applying the localization theorem \ref{theorem : localization} to $\beta$, we have
\begin{displaymath}\label{equation : integration of beta}
    \begin{array}{lll}
        \displaystyle \int_M \beta & = & \sum_{z \in M^{S^1}} \frac{(\alpha|_z)^2 \cdot [\omega_{H_1}]|_z \cdot [\omega_{H_2}]|_z \cdots [\omega_{H_{n-2k-1}}]|_z}{e^{S^1}(z)} = 0 \\
    \end{array}
\end{displaymath}
by dimensional reason.
Since $\alpha \neq 0$ in $H^{2k}_{S^1}(M)$, there exists some fixed point $z \in M^{S^1}$ such that $\alpha|_z \neq 0$ by Kirwan's injectivity theorem \ref{theorem : Kirwan injectivity}. Now, let's determine a sign of each summand of the integration above.
For every $z \in I_1$ with $\alpha|_z \neq 0$, the sign of $e^{S^1}(z)$ equals to the sign of $(-1)^k$ since $\alpha$ vanishes on $\mathcal{P}_1$ so that $\mathrm{ind}(z)$ must be of $2k-4i$ for some positive integer $i$. Also by Proposition \ref{proposition : equivariant cohomology class}, we can easily see that
\begin{displaymath}
    \begin{array}{l}
        (\alpha|_z)^2 \cdot [\omega_{H_1}]|_z \cdot [\omega_{H_2}]|_z \cdots [\omega_{H_{n-2k-1}}]|_z \\
         = (\alpha|_z)^2(-H_1(z))(-H_2(z))\cdots (-H_{n-2k-1}(z)) > 0\\
    \end{array}
\end{displaymath}
since $H_i(z) < 0 $ for every $z \in I_1$. Hence we have
$$ (-1)^k \cdot \sum_{z \in I_1} \frac{(\alpha|_z)^2 \cdot [\omega_{H_1}]|_z \cdot [\omega_{H_2}]|_z \cdots [\omega_{H_{n-2k-1}}]|_z}{e^{S^1}(z)} \geq 0.$$
The equality holds if and only if $\alpha|_z = 0$ for every $z \in I_1$.
For $z \in I_j$ with $\alpha|_z \neq 0$ and $j = 2, \cdots, n-2k-2$, the sign of $e^{S^1}(z)$ equals to the sign of $(-1)^{k+j-1}$ since $\mathrm{ind}(z) = 2k + 2j - 2$.
Also, we have $H_i(z) > 0 $ for $i < j$ and $H_i(z) < 0$ for $i \geq j$ so that the numerator has a sign of $(-1)^{j-1}$. Therefore,
$$ (-1)^k \cdot \sum_{z \in I_j} \frac{(\alpha|_z)^2 \cdot [\omega_{H_1}]|_z \cdot [\omega_{H_2}]|_z \cdots [\omega_{H_{n-2k-1}}]|_z}{e^{S^1}(z)} \geq 0$$
for all $j=2, \cdots, n-2k-2$. And the equality holds if and only if $\alpha|_z = 0$ for every $z \in I_j$.
For $z \in I_{n-2k-1}$ with $\alpha|_z \neq 0$, the index of $z$ is either $2n-2k-4$ or $2n-2k$ since $\alpha$ vanishes on $\mathcal{P}_3 = \Lambda_{2n-2k-2}$.
Hence the sign of $e^{S^1}(z)$ equals to $(-1)^{n-k}$.
Also, $H_i(z) > 0 $ for every $i < n-2k-1$ so that the sign of numerator equals to the sign of $(-1)^{n-2k-2}$.
Therefore, we have
$$ (-1)^k \cdot \sum_{z \in I_{n-2k-1}} \frac{(\alpha|_z)^2 \cdot [\omega_{H_1}]|_z \cdot [\omega_{H_2}]|_z \cdots [\omega_{H_{n-2k-1}}]|_z}{e^{S^1}(z)} \geq 0.$$
The equality holds if and only if $\alpha|_z = 0$ for every $z \in I_{n-2k-1}$.
Finally, for $z \in I_{n-2k}$ with $\alpha|_z \neq 0$, the sign of $e^{S^1}(z)$ should be $(-1)^{n-k-1}$ since $\alpha$ vanishes on $\mathcal{P}_2$ so that $\mathrm{ind}(z)$ is $2n-2k+2$.
Also, $H_i(z) > 0$ for every $i$ so that the sign of numerator is $(-1)^{n-2k-1}$. Therefore
$$ (-1)^k \cdot \sum_{z \in I_{n-2k}} \frac{(\alpha|_z)^2 \cdot [\omega_{H_1}]|_z \cdot [\omega_{H_2}]|_z \cdots [\omega_{H_{n-2k-1}}]|_z}{e^{S^1}(z)} \geq 0.$$
The equality holds if and only if $\alpha|_z = 0$ for every $z \in I_{n-2k}$.

To sum up, every summand has the same sign, and there is at least one summand which is non-zero. Hence the integral $\int_M \beta$ cannot be zero which leads a contradiction.
This finishes the proof.

\end{proof}

\end{document}